\documentclass[10pt]{amsart}

\usepackage[utf8]{inputenc}
\usepackage{amsthm, amsmath, amsfonts}
\usepackage[a4paper]{geometry}
\usepackage{hyperref}
\usepackage[abbrev, nobysame]{amsrefs}

\newcommand\NN{\mathbb N}

\newcommand\Prob{\mathbf P}

\newtheorem*{lemma}{Lemma}
\newtheorem*{theorem}{Theorem}

\begin{document}

\title{The probability that two random integers are coprime}

\author[J. Bureaux]{Julien Bureaux}
\address[J. Bureaux]{MODAL'X, Université Paris Nanterre, 200 avenue de la République, 92001 Nanterre}
\email{julien.bureaux@math.cnrs.fr}

\author[N. Enriquez]{Nathana\"el Enriquez}
\address[N. Enriquez]{MODAL'X, Université Paris Nanterre, 200 avenue de la République, 92001 Nanterre}
\email{nathanel.enriquez@u-paris10.fr}
\address[N. Enriquez]{LPMA, Université Pierre et Marie Curie, 4 place Jussieu, 75005 Paris}

\begin{abstract}
An equivalence is proven between the Riemann Hypothesis and the speed of convergence to $6/{\pi^2}$ of the probability that two independent random variables following the same geometric distribution are coprime integers, when the parameter of the distribution goes to 0.

\end{abstract}

\maketitle

%
What is the probability that two random integers are coprime?
Since there is no uniform probability distribution on $\NN$, this question must be made more precise. Before going into probabilistic considerations, we would like to remind a seminal result of Dirichlet~\cite{dirichlet_1889_bestimmung}: the set  $\mathcal P$ of ordered pairs of coprime integers has an asymptotic density in $\NN^2$ which is equal to $6/\pi^2 \approx 0.608$. The proof is not trivial and can be found in the book of Hardy and Wright (\cite{hardy_2008_introduction}, Theorem~332), but if one admits the existence of a density $\delta$, its value can be easily obtained. Indeed, the set  $\NN^2$ is the disjoint union of $\mathcal P$, $2\mathcal P$, $3\mathcal P$, \ldots which have respective density $\delta$, $\delta/2^2$, $\delta/3^2$, \ldots, hence
$$
1=\delta+\frac\delta{2^2}+ \frac\delta{3^2}+\cdots= \delta \sum_{k=1}^{+\infty} \frac{1}{k^2} = \delta \frac{\pi^2}{6}.
$$ 

A natural choice for picking two integers at random is to draw two random variables $X$ and $Y$ independently according to the uniform distribution on $\{1,2,\dots,n\}$ and let $n$ goes to $+\infty$. In this case,  denoting by $\varphi(n)$ the classical Euler totient function, the probability that $X$ and $Y$ are coprime equals
  \begin{equation}
    \label{eq:s}
    \frac1{n^2} \left(2\sum_{k=1}^n\varphi(k) -1\right)=\frac6{\pi^2}+O(\frac{\log n}n),
  \end{equation}
as was shown by Mertens \cite{mertens_1874_asymptotische} in 1874. The sharpest error term in this problem up to now was obtained in 1963 by Walfisz~\cite{walfisz_1963_weylsche} who replaced the $\log n$ term of Mertens by $(\log n)^{2/3}(\log\log n)^{4/3}$. One can quickly realize that the error term cannot boil down to $o(1/n)$. Indeed, the passage from $n$ to $n+1$  perturbates the average by a term of order $1/n$. A striking result by Montgomery~\cite{montgomery_1987_fluctuations} shows  that the fluctuations are actually at least of order $\sqrt{\log \log n}/n$. 

Recently, Kaczorowski and Wiertelak considered in \cite{KW1} a smoothing of the sequence ~\eqref{eq:s} above which corresponds to another family of probabilities. Namely, for a given $n$, the former uniform distribution on $\{1,2,\dots,n\}^2$ is replaced by the probability distribution  putting a weight $\log(\frac{n}{i+j})$ to the couple $(i,j)$ for all integers $i,j$ satisfying $2\leq i+j\leq n$, up to a normalization factor. Notice that, under this distribution, the two random variables $X$ and $Y$ have the same distribution but are not independent anymore. Simultaneously, they iterated this smoothing procedure in \cite{KW2}, and smartly recovered Montgomery's result. 
 
In this paper, we consider a natural parametrized family of probability distributions for $X$ and $Y$ that smoothes even more the fluctuations with respect to the parameter. For $\beta>0$, the random variables $X$ and $Y$ are drawn independently according to the geometric distribution of parameter $1-e^{-\beta}$:
\[
  \forall n  \in \NN, \qquad \Prob[X=n] = (1-e^{-\beta})e^{-\beta(n-1)}.
\]
The asymptotic regime corresponding to large integers is obtained by taking $\beta\to 0$. For this distribution, we are able to obtain a direct connection with the Riemann Hypothesis. 

\begin{theorem}
  Let $X,Y$ be two independent random variables following the geometric distribution of parameter $1-e^{-\beta}$.
  The Riemann Hypothesis holds if and only if, for all $\epsilon > 0$, as $\beta$ goes to $0$, 
  \begin{equation}
    \label{eq:h}
         \Prob[\gcd(X,Y) = 1] = \frac{6}{\pi^2} + \frac{6}{\pi^2}\beta + O(\beta^{\frac32-\epsilon}).
  \end{equation}
\end{theorem}
Before we start with the proof, let us make a comment. The ``smoothing" we have chosen here is stronger than the smoothing of \cite{KW1} we described above. 
However, one could think about an even stronger one, corresponding to independent variables $X$ and $Y$ both following the Zeta distribution of parameter $\alpha > 1$, i.e. 
$$\forall n\in \NN,\quad \Prob[X=n]=\Prob[Y=n]=\frac1{\zeta(\alpha)}\frac1{n^\alpha}.$$
In this case, $\gcd(X,Y)$ is distributed according to  the Zeta distribution of parameter $2\alpha$. This explicit example shows that a too strong smoothing erases all oscillations of the probability of coprimality and makes the dependence  on the localization of the zeroes of the zeta function vanish.\medbreak

Now, we turn to the proof of the theorem. We begin by expressing the probability that $X$ and $Y$ are coprime as a function of~$\beta$:
\[
  \Prob[\gcd(X,Y)=1] = \sum_{(x,y) \in \mathcal P} \Prob[X=x,Y=y] = (e^{\beta}-1)^2 \sum_{(x,y) \in \mathcal P} e^{-\beta(x+y)}.
\]
We see therefore that the function $f \colon (0,+\infty) \to (0,+\infty)$ defined for all $\beta > 0$ by
\[
  f(\beta) = \sum_{(x,y) \in \mathcal P} e^{-\beta(x+y)}
\]
plays a key role in the analysis of this probability. 
In particular, condition~\eqref{eq:h} is equivalent to the following expansion of $f(\beta)$ as $\beta$ goes to $0$:
\begin{equation}
  \label{eq:h'}
  \tag{2'}
  \forall \epsilon > 0,\qquad f(\beta) = \frac{6}{\pi^2}\frac{1}{\beta^2} + O(\frac{1}{\beta^{\frac 1 2 + \epsilon}}).
\end{equation}
In the sequel, we will work with \eqref{eq:h'} instead of \eqref{eq:h}.

\section{From the probability estimate to Riemann's Hypothesis.}

The starting point of the proof is the following relation between the function $f$ and the Riemann zeta function.
\begin{lemma}
  For all complex number $s$ such that $\Re(s) > 2$,
\begin{equation}
  \label{eq:mellin}
  \int_0^\infty f(\beta)\beta^{s-1}d\beta = \Gamma(s)\frac{\zeta(s-1)-\zeta(s)}{\zeta(s)}.
\end{equation}
\end{lemma}

\begin{proof}
  Let us first consider the case of a real number $s > 2$. By using the Fubini-Tonelli theorem,
  \[
    \int_0^\infty f(\beta)\beta^{s-1}d\beta = \sum_{(x,y) \in \mathcal P} \int_0^\infty e^{-\beta(x+y)} \beta^{s-1}d\beta
    = \sum_{(x,y) \in \mathcal P} \frac{\Gamma(s)}{(x+y)^s}.
  \]
  To simplify this series, we observe that each $(x,y) \in \NN^2$ can be written as $(nx',ny')$ with $n \in \NN$ and $\gcd(x',y') = 1$ in a unique way, leading to the identity
  \[
   \sum_{(x,y) \in \NN^2} \frac{1}{(x+y)^s} = \Biggl(\sum_{n=1}^\infty \frac 1{n^s}\Biggr)\Biggl(\sum_{(x,y) \in \mathcal P} \frac1{(x+y)^s}\Biggr).
  \]
  Since $s > 2$, the left-hand side is easily expressed in terms of the zeta function:
  \[
    \sum_{(x,y) \in \NN^2} \frac1{(x+y)^s} = \sum_{n=1}^\infty \frac{n-1}{n^s} = \zeta(s-1) - \zeta(s),
  \]
  which ends the proof of \eqref{eq:mellin}.
    The extension of \eqref{eq:mellin} to the complex domain $\Re(s) > 2$ follows the same steps, except that permutations of the order of summation are justified by the Fubini theorem.
\end{proof}

Now, let us assume that \eqref{eq:h'} holds. Our strategy is to show that the function $\Delta$ defined for all $s$ with $\Re(s) > 2$ by
\[
  \Delta(s) = \Gamma(s)\frac{\zeta(s-1)-\zeta(s)}{\zeta(s)} - \frac{6}{\pi^2} \frac{1}{s-2}
\]
has a holomorphic continuation to the domain $\Re(s) > \frac 1 2$. This fact would indeed prevent $\zeta$ from having zeros with real part larger than $\frac12$, since such zeros would correspond to poles of $\Delta$ in the same region.
The celebrated functional equation (Theorem 2.1 in \cite{titchmarsh_1986_riemann}) satisfied by $\zeta$ shows that its zeros in the critical strip $0 < \Re(s) < 1$ are symmetrically distributed with respect to the vertical line $\Re(s) = \frac 12$. Hence, the function $\zeta$ would not vanish in the region $0 < \Re(s) < \frac 12$ either.

The holomorphic continuation is obtained by an integral representation of $\Delta$ based on formula~\eqref{eq:mellin} in the Lemma. Namely, for all $s$ with $\Re(s) > 2$,
\[
  \Delta(s) = \int_0^1 \left(f(\beta) - \frac6{\pi^2}\frac1{\beta^2}\right)\beta^{s-1}d\beta + \int_1^\infty f(\beta)\beta^{s-1}d\beta.
\]
The estimate \eqref{eq:h'} for $f(\beta)$ in the neighbourhood of $0$ implies that the first integral defines a holomorphic function of $s$ in the domain $\Re(s) > \frac12$.
The second integral defines a holomorphic function of $s$ in the whole complex plane because $f(\beta) = O(e^{-2\beta})$ when $\beta$ goes to $+\infty$.

\section{From Riemann's Hypothesis to the probability estimate.}

Formula \eqref{eq:mellin} shows that the function $s \mapsto \Gamma(s)(\zeta(s-1)-\zeta(s))/\zeta(s)$ is the Mellin transform of the function $f$ in the domain $\Re(s) > 2$.
An application of the Mellin inversion formula leads therefore to the following integral representation of $f(\beta)$ for all $\beta > 0$ and $c > 2$,
\begin{equation*}
   f(\beta) = \lim_{T \to +\infty} \frac{1}{2i\pi}\int_{c-iT}^{c+iT} \Gamma(s)\frac{\zeta(s-1)-\zeta(s)}{\zeta(s)\beta^s}ds.
 \end{equation*}

 Let $\epsilon > 0$ and $T > 0$. Recall here that the zeta function has a unique holomorphic continuation to $\mathbb C\setminus\{1\}$, and that $1$ is a simple pole with residue $1$.
 Assuming that Riemann's Hypothesis holds, the only pole of the function $s \mapsto \Gamma(s)(\zeta(s-1)-\zeta(s))/\zeta(s)$ in the region $\Re(s) > \frac 12$ lies at $s = 2$ with residue $\Gamma(2)/(\zeta(2)\beta^2)$. We can therefore apply the residue theorem to show that
 \begin{align*}
  \label{eq:residue}
   \frac{6}{\pi^2\beta^2} =& \frac{1}{2i\pi}\int_{3-iT}^{3+iT} \Gamma(s)\frac{\zeta(s-1)-\zeta(s)}{\zeta(s)\beta^s}ds\\
&  + \frac{1}{2i\pi}\left(\int_{3+iT}^{\frac12 + \epsilon +iT} + \int_{\frac12 + \epsilon -iT}^{3-iT}\right) \Gamma(s)\frac{\zeta(s-1)-\zeta(s)}{\zeta(s)\beta^s}ds\\
 & + \frac{1}{2i\pi} \int_{\frac 12 + \epsilon+iT}^{\frac 12 + \epsilon -iT} \Gamma(s)\frac{\zeta(s-1)-\zeta(s)}{\zeta(s)\beta^s}ds.
 \end{align*}
 Taking $c = 3$ in the Mellin inversion formula above, we see that the integral on $[3-iT;3+iT]$ tends to $f(\beta)$ as $T$ tends to $+\infty$. In order to deal with the other three integrals, we are going to use the following estimates as $|t| \to +\infty$:
 \begin{itemize}
   \item As a consequence of the complex Stirling formula, the $\Gamma$ function has exponential decay along vertical lines. In particular,
     \[
       \sup_{\sigma \in [\frac12;3]}|\Gamma(\sigma+it)| = O(|t|^{\frac 52}e^{-\frac \pi 2 |t|}).
     \]
   \item The zeta function has polynomial growth along vertical lines (see Section 5.1 of \cite{titchmarsh_1986_riemann}). There exists $k > 0$ such that 
     \[
       \sup_{\sigma \geq -\frac12} |\zeta(\sigma + it)| = O(|t|^k).
     \]
   \item The inverse zeta function has sub-polynomial growth along vertical lines in the region $\Re(s) > \frac 12$ as a consequence of Riemann's Hypothesis (see Theorem 14.2 in \cite{titchmarsh_1986_riemann}). For all $\alpha >  0$,
 \[
   \sup_{\sigma \in [\frac12+\epsilon;3]}\frac{1}{|\zeta(\sigma + it)|} = O(|t|^\alpha).
 \]
 \end{itemize}
We can see with these bounds that both integrals along the horizontals segments $[3+iT;\frac12 + \epsilon + iT]$ and $[\frac12 + \epsilon - iT; 3-iT]$ tend to $0$ as $T \to +\infty$. Similarly, the integral along the vertical segment $[\frac12+\epsilon+iT;\frac12+\epsilon-iT]$ is bounded by $O(\beta^{-\frac 12-\epsilon})$.


\end{document}